\IfFileExists{pdfmanagement.sty}
  {\RequirePackage{pdfmanagement}}
  {}
\RequirePackage{silence}
\documentclass[11pt,leqno,noamsfonts]{amsart}

\usepackage[p,scosf,looser]{scholax}
\usepackage{mathtools}
\usepackage[scaled=1.075,ncf,vvarbb]{newtxmath}
\usepackage{microtype}
\usepackage{enumitem}
\usepackage{booktabs}
\usepackage[dvipsnames]{xcolor}
\usepackage{tikz-cd}
\usepackage{graphicx}
\usepackage{todonotes}
\usepackage{zref-clever}
\usepackage[pagebackref,colorlinks]{hyperref}

\hypersetup{
  citecolor=Green,
  urlcolor=violet,
  pdftitle={Bertini theorems for Hilbert–Samuel multiplicity over finite fields},
  pdfauthor={Rahul Ajit and Matthew Bertucci},
  }
\IfFileExists{pdfmanagement.sty}
  {}
  {\hypersetup{linkcolor=Red}}

\zcsetup{S,nameinlink=false,lastsep={, and\nobreakspace}}

\zcRefTypeSetup{section}{
Name-sg = \S ,
name-sg = \S ,
Name-pl = \S ,
name-pl = \S ,
namesep = \, ,
}
\zcRefTypeSetup{claim}{
Name-sg = Claim ,
name-sg = claim ,
Name-pl = Claims ,
name-pl = claims ,
}

\ifdefined\newcounteralias
  \newcommand{\newaliastheorem}[3]{%
    \newtheorem{#1}[#2]{#3}%
    }
\else
  \usepackage{aliascnt}
  \newcommand{\newaliastheorem}[3]{%
    \newaliascnt{#1}{#2}%
    \newtheorem{#1}[#1]{#3}%
    \aliascntresetthe{#1}%
    }
\fi

\newcommand{\bP}{\mathbb{P}}
\newcommand{\Fq}{\mathbb{F}_q}

\newcommand{\cP}{\mathcal{P}}

\newcommand{\red}{\mathrm{red}}

\newcommand{\afrak}{\mathfrak{a}}
\newcommand{\mfrak}{\mathfrak{m}}
\newcommand{\nfrak}{\mathfrak{n}}
\newcommand{\pfrak}{\mathfrak{p}}

\newcommand{\sheaf}{\mathscr}

\DeclareMathOperator{\Spec}{Spec}

\DeclareMathOperator{\gr}{gr}
\DeclareMathOperator{\ord}{ord}

\DeclareMathOperator{\Sing}{Sing}
\DeclareMathOperator{\Reg}{Reg}
\DeclareMathOperator{\edim}{edim}
\DeclareMathOperator{\Sym}{Sym}

\numberwithin{equation}{section}

\newtheorem{theorem}{Theorem}[section]
\newaliastheorem{lemma}{theorem}{Lemma}
\newaliastheorem{corollary}{theorem}{Corollary}
\newaliastheorem{proposition}{theorem}{Proposition}

\theoremstyle{definition}
\newaliastheorem{definition}{theorem}{Definition}
\newaliastheorem{claim}{theorem}{Claim}

\theoremstyle{remark}
\newaliastheorem{remark}{theorem}{Remark}

\title[Bertini theorems for multiplicity over finite fields]{Bertini theorems for Hilbert--Samuel multiplicity over finite fields}

\author{Rahul Ajit}
\address{Department of Mathematics, The University of Utah, Salt Lake City, UT 84112}
\email{rahulajit@math.utah.edu, confusedrahul@yahoo.com}

\author{Matthew Bertucci}
\address{Department of Mathematics, Willamette University, Salem, OR 97301}
\email{mbertucci@willamette.edu}

\date{\today}

\dedicatory{To Prof.\ Heisuke Hironaka, in memoriam.}

\begin{document}
\begin{abstract}
Let $X$ be a quasiprojective subscheme of $\bP^n_{\Fq}$.
We prove a Bertini theorem for Hilbert--Samuel multiplicity of $X$; that is, there exists a positive-density set of hypersurfaces $H_f$ such that for every point $\xi\in X\cap H_f$, one has $\ord_\xi(f)=1$ and $e_\xi(X\cap H_f)=e_\xi(X)$.

Furthermore, we extend this result on hypersurfaces to complete intersections, hypersurfaces containing a prescribed subscheme, and semiample linear systems.
\end{abstract}

\maketitle

\tableofcontents
\newpage
\section{Introduction}

Let $k$ be a field and let $X$ be a quasiprojective subscheme of $\bP^n_k$.
For $\xi\in X$, the Hilbert--Samuel multiplicity $e_\xi(X)$ provides a coarse measure of the singularity of $X$ at $\xi$.
For example, if $X$ is regular at $\xi$, then $e_\xi(X)=1$; if $X$ is reduced and equidimensional, the converse holds as well (\cite[Theorem~40.6]{nagata}).

One of the classical Bertini theorems says that if $k$ is algebraically closed and $X$ is smooth, then a general hyperplane section of $X$ is also smooth (see, e.g., \cite{kleiman:bertini}).
If $k$ is infinite but not necessarily algebraically closed, then one can still always find infinitely many smooth hyperplane sections of $X$.

When $k$ is finite, the classical statement fails: one can construct smooth $X$ so that \emph{no} hyperplane sections, in fact no hypersurface sections up to a fixed degree $d$, are smooth (\cite[Question~10]{katz}, \cite[Theorem~3.1]{poonen:bertini}).
However, Poonen showed in \cite{poonen:bertini} that if one considers hypersurfaces of degree $d$ in $\bP^n_k$ as $d\to\infty$, a positive-density subset do give smooth sections.
The closed-point sieve introduced in \cite[Theorems~1.1--1.3]{poonen:bertini} has been extended to complete intersections (\cite[Theorem~1.2]{bucur-kedlaya}), hypersurfaces containing a prescribed subscheme (\cite[Theorem~1.1]{poonen:subvar}, \cite[Theorem~1.1]{gunther}, \cite[Theorem~1.2]{wutz2016}), simultaneous conditions on disjoint smooth loci (\cite[Proposition~4.8]{ghosh-krishna}), semiample systems (\cite[Theorem~9.10]{erman-wood}), and locally free first-order Taylor conditions (\cite[Theorems~A--C]{Matthew}).

Over an algebraically closed field of characteristic zero and with $X$ either irreducible or both reduced and equidimensional, the analogous classical Bertini theorem for Hilbert--Samuel multiplicity holds (see \cite[Proposition~4.5]{tommaso} and \cite[Proposition~3.4]{complex-links}).
That is, for a general hyperplane $H$, we have $e_\xi(X\cap H)=e_\xi(X)$ for all $\xi\in X\cap H$.
In \zcref{sec:infinite-field}, we use the results of \cite{spreafico} and \zcref{sec:bridge} to prove a generalization of this Bertini theorem to any infinite base field whenever $X_{\red}$ admits a smooth permissible
stratification, without assuming that $X$ is reduced or equidimensional.
Over a perfect field, such a stratification exists by \zcref{thm:stratification}.

The main result of this paper is the following Bertini theorem for Hilbert--Samuel multiplicity over finite fields.

\begin{theorem}
\label{thm:intro-main}
Let $X\subseteq\bP^n_{\Fq}$ be quasiprojective, and let
\[
  X_{\red}=Y_1\sqcup\cdots\sqcup Y_N
\]
be a smooth permissible stratification.
Set $s_i=\dim Y_i$ and let $\cP$ be the set of homogeneous polynomials $f$ such that $Y_i\cap H_f$ is smooth of dimension $s_i-1$ for every $i$, where dimension $-1$ means the empty scheme.
Then $\cP$ has density
\begin{equation}
\label{eq:intro-density}
  \mu(\cP)=\prod_{i=1}^N\zeta_{Y_i}(s_i+1)^{-1}>0.
\end{equation}
Furthermore, if $f\in\cP$ and $\xi\in X\cap H_f$, then
\begin{gather*}
  \ord_\xi(f)=1,
  \qquad
  f_\xi\text{ is a nonzerodivisor on }\sheaf O_{X,\xi},
  \\
  \text{and}
  \quad
  e_\xi(X\cap H_f)=e_\xi(X).
\end{gather*}
\end{theorem}

Since we prove that such a stratification always exists (see \zcref{thm:stratification}), the second part of the theorem implies that a positive density of hypersurface sections preserves multiplicity at all points.

Similar local arguments apply to the closed-point sieves in \cite[Theorem~1.2]{bucur-kedlaya}, \cite[Theorem~1.1]{poonen:subvar}, \cite[Theorem~1.1]{gunther}, \cite[Theorems~1.1 and 1.2]{wutz2016}, and \cite[Theorem~9.10]{erman-wood}.
The resulting complete-intersection, containing a given subscheme, and semiample statements are proved in \zcref{thm:bk-mult,thm:gw-mult,thm:semiample-mult}, respectively.

\section{Acknowledgments}

We would like to thank our PhD advisors --- for the first-named author, Christopher Hacon and Karl Schwede, and for the second-named author, Sean Howe --- for their constant encouragement, unwavering support, inspiring teachings, and infinite patience.
The first-named author is grateful to Ilya Smirnov for pointing out \cite[Ch.\,0, (2.1.1)]{bennett}, which played a crucial role in this paper.
The first-named author would like to thank Daniel Apsley, Vasudevan Srinivas, and Vijaylaxmi Trivedi for their interest in this project and for many interesting discussions.
Rahul Ajit was partially supported by a Summer Research Fellowship at the University of Utah, NSF research grants \href{https://www.nsf.gov/awardsearch/show-award?AWD_ID=2301374}{DMS-2301374} and \href{https://www.nsf.gov/awardsearch/show-award?AWD_ID=2501903}{DMS-2501903}, and by a grant from the Simons Foundation, SFI-MPS-MOV-00006719-07, while working on this project.
Matthew Bertucci was partially supported during the preparation of this work by the University of Utah's NSF Research Training Grant \href{https://www.nsf.gov/awardsearch/showAward?AWD_ID=1840190}{\#1840190}.

\section{Preliminaries}
\label{sec:prelim}

We begin by fixing notation and providing definitions of the necessary concepts from commutative algebra.

Let
\[
  S_d=H^0(\bP^n_{\Fq},\sheaf O_{\bP^n}(d))
  \qquad
  \text{and}
  \qquad
  S_{\mathrm{homog}}=\bigcup_{d\ge0}S_d.
\]
The density of a subset $\cP\subseteq S_{\mathrm{homog}}$ is
\[
  \mu(\cP)=\lim_{d\to\infty}\frac{\#(\cP\cap S_d)}{\#S_d},
\]
when the limit exists.

For a scheme $X$ of finite type over a field $k$, write $|X|$ for its closed points.
For $\xi\in |X|$, set $\deg(\xi)=[\kappa(\xi):k]$.

When $k=\Fq$, the zeta function of $X$ is
\[
 \zeta_X(s)=\prod_{\xi\in|X|}
 \bigl(1-q^{-s\deg(\xi)}\bigr)^{-1}
\]
which is defined when $\Re(s)>\dim X$.

\begin{definition}[{\cite[\S\,14]{matsumura}}]
Let $(R,\mfrak)$ be a Noetherian local ring.
The \emph{Hilbert--Samuel function} of $R$ is
    \[ \chi_R(\ell) = \lambda_R(R/\mfrak^\ell)\]
where $\lambda_R$ denotes length as an $R$-module.
There is a polynomial $P_R$, the \emph{Hilbert--Samuel polynomial} of $R$, that agrees with $\chi_R$ for $\ell\gg 0$, and has degree $\dim R$.
Define the \emph{(Hilbert--Samuel) multiplicity} $e(R)$ of $R$ to be $(\dim R)!$ times the leading coefficient of $P_R$.

For a Noetherian scheme $X$ and point $\xi\in X$, the \emph{multiplicity} $e_\xi(X)$ of $X$ at $\xi$ is the multiplicity of the local ring $\sheaf{O}_{X,\xi}$.
\end{definition}

\begin{definition}
Let $(R,\mfrak)$ be a Noetherian local ring and $\afrak\subseteq R$ an ideal.
The \emph{associated graded ring} of $\afrak$ is
  \[
    \gr_{\afrak}(R) = \bigoplus_{\nu\geq 0} \afrak^\nu/\afrak^{\nu+1}
  \]
where by convention, $\afrak^0=R$.
When $\afrak=\mfrak$, we call $\gr_{\mfrak}(R)$ the associated graded ring of $R$.

For $f\in R$, set $\ord_{\afrak}(f)=\sup\{\nu\mid f\in\afrak^\nu\}$.
(By the Krull intersection theorem, this is finite for $f\neq 0$.)
The \emph{initial form} $f^*$ of $f$ is the image of $f$ in $\gr_{\afrak}(R)$, i.e., the element $f^*=f+\afrak^{\ord_{\afrak}(f)+1}$.

If $I\subseteq R$ is an ideal, define $I^*$ to be the homogeneous ideal of $\gr_{\afrak}(R)$ generated by the initial forms of elements of $I$.
\end{definition}

\begin{remark}
If $I=(f_1,\dots,f_c)$, then in general we only have $(f_1^*,\dots,f_c^*)\subset I^*$, not equality.
In the principal case $I=(f)$, if $R$ is regular, then $(f)^*=(f^*)$ but this is not true in general; for a detailed account of the relationship between $I$ and $I^*$, see \cite{init-forms} and \cite{eq-tangent-cones}.
\end{remark}

For a Noetherian scheme $X$, $\xi\in X$, and $f\in\sheaf O_{X,\xi}$, write
\[
  \ord_\xi(f)=\ord_{\mfrak_{X,\xi}}(f).
\]
Multiplication by a unit does not change this order, hence it is defined for a local equation of an effective Cartier divisor or of the zero scheme of a section of an invertible sheaf.

\begin{lemma}[{cf.\ \cite[Ch.\,0, (1.1.1)]{bennett}}]
\label{lem:quotient-surj}
Let $(R,\mathfrak m)$ be a Noetherian local ring and let $I\subseteq\mathfrak m$ be an ideal.
The quotient map $R\to R/I$ induces a canonical graded isomorphism
\[
 \gr_\mathfrak m(R)/I^*
 \xrightarrow{\ \sim\ }
 \gr_{\mathfrak m/I}(R/I).
\]
\end{lemma}

\begin{proof}
In degree $n$, the induced map is
\[
 \frac{\mathfrak m^n}{\mathfrak m^{n+1}}
 \longrightarrow
 \frac{\mathfrak m^n+I}{\mathfrak m^{n+1}+I},
\]
and its kernel is
\[
 \frac{(I\cap\mathfrak m^n)+\mathfrak m^{n+1}}{\mathfrak m^{n+1}}.
\]
This is the degree-$n$ component of $I^*$.
Indeed, if $x\in I\cap\mathfrak m^n$, then its degree-$n$ class is either zero or $x^*$.
Conversely, a homogeneous element of $(I^*)_n$ is a sum of products $a^*x^*$ with $x\in I$ and
$\operatorname{ord}_\mathfrak m(a)+\operatorname{ord}_\mathfrak m(x)=n$; each product is represented by $ax\in I\cap\mathfrak m^n$.
Thus the kernel is $I^*$ degree by degree.
\end{proof}

\begin{lemma}
\label{lem:initial-regular-sequence}
Let $(R,\mfrak)$ be a Noetherian local ring, let $f_1,\ldots,f_c\in\mfrak$, and set
$J=(f_1,\ldots,f_c)$.
Suppose that each $f_i$ has $\mfrak$-adic order one and that
$f_1^*,\ldots,f_c^*$ is a regular sequence on
$G=\gr_\mfrak(R)$.
Then:
\begin{enumerate}[label=\textup{(\alph*)}]
\item $f_1,\ldots,f_c$ is an $R$-regular sequence;
\item $J^*=(f_1^*,\ldots,f_c^*)$;
\item the canonical map induces an isomorphism
\[
 G/(f_1^*,\ldots,f_c^*)
 \xrightarrow{\ \sim\ }
 \gr_{\mfrak/J}(R/J);
\]
\item $e(R/J)=e(R)$.
\end{enumerate}
\end{lemma}

\begin{proof}
The isomorphism in (c) is the degree-one case of \cite[Theorem~2.3]{init-forms}.
We include the elementary one-element argument, then apply the result successively.

So, $c=1$.
Write $f=f_1$.
By the Krull intersection theorem,
$\bigcap_{n\geq 0}\mathfrak m^n=0$
(\cite[\href{https://stacks.math.columbia.edu/tag/00IP}{Tag~00IP}]{stacks-project}); hence every nonzero element of $R$ has finite $\mathfrak m$-adic order.
Since $f^*$ is a nonzerodivisor on $G$, for every $0\neq r\in R$,
\[
 \operatorname{ord}_\mathfrak m(fr)
 =1+\operatorname{ord}_\mathfrak m(r).
\]
In particular, $f$ is a nonzerodivisor.
The same equality gives, for every $n\geq1$,
\[
 (f)\cap\mathfrak m^n=f\mathfrak m^{n-1}.
\]
Consequently $(f)^*=(f^*)$, and \zcref{lem:quotient-surj} gives
\[
 \gr_{\mathfrak m/(f)}(R/(f))
 \simeq G/(f^*).
\]
The image of $f_2^*$ is therefore the initial form of the image of $f_2$ in $R/(f)$ and is a nonzerodivisor on the associated graded ring of that quotient.
After $c$ steps we obtain
\[
 G/(f_1^*,\ldots,f_c^*)\cong\gr_{\mfrak/J}(R/J).
\]
Comparing the kernel of $G\to\gr_{\mfrak/J}(R/J)$ with \zcref{lem:quotient-surj} gives
$J^*=(f_1^*,\ldots,f_c^*)$.
This proves (a)--(c).

For completeness, put
\[
 h_R(n)=\ell_R(\mathfrak m^n/\mathfrak m^{n+1}),
 \qquad h_R(-1)=0.
\]
The exact sequence defined by multiplication by the degree-one nonzerodivisor $f^*$ gives
\[
 h_{R/(f)}(n)=h_R(n)-h_R(n-1)
 \qquad(n\geq0).
\]
Thus, with
$\chi_R(n)=\ell_R(R/\mathfrak m^{n+1})$ and $\chi_R(-1)=0$,
\[
 \chi_{R/(f)}(n)
 =\chi_R(n)-\chi_R(n-1).
\]
Taking Hilbert polynomials shows $e(R/(f))=e(R)$; induction gives (d).
\end{proof}

\begin{corollary}\label{cor:hilbert-samuel-finite-differences}
In the notation of \zcref{lem:initial-regular-sequence}, define
\[
 \chi_R(n)=\ell_R(R/\mathfrak m^{n+1}),
 \qquad \chi_R(n)=0\quad(n<0),
\]
and $(\Delta a)(n)=a(n)-a(n-1)$.
Then, for every $n\geq0$,
\[
 \chi_{R/J}(n)=\Delta^c\chi_R(n).
\]
Equivalently,
\[
 \sum_{n\geq0}
 \ell_R\!\left(\frac{(\mathfrak m/J)^n}{(\mathfrak m/J)^{n+1}}\right)t^n
 =(1-t)^c
 \sum_{n\geq0}
 \ell_R\!\left(\frac{\mathfrak m^n}{\mathfrak m^{n+1}}\right)t^n.
\]
If $d=\dim R$, then
\[
 e_j(R/J)=e_j(R)\qquad(0\leq j\leq d-c).
\]
\end{corollary}

\begin{proof}
The Hilbert-series identity follows from \zcref{lem:initial-regular-sequence}(c), since each $f_i^*$ has degree one.
Taking cumulative sums gives the finite-difference formula.
For $n\gg0$, write
\[
 \chi_R(n)=\sum_{j=0}^d(-1)^j e_j(R)
 \binom{n+d-j}{d-j}.
\]
Applying $\Delta^c$ and using
$\Delta\binom{n+a}{a}=\binom{n+a-1}{a-1}$
proves the assertion about the Hilbert coefficients.
\end{proof}

\begin{lemma}[{\cite[\href{https://stacks.math.columbia.edu/tag/00NQ}{Tag~00NQ}]{stacks-project}}]
\label{lem:smooth-first-jets}
Let $Y$ be regular, let $\xi\in Y$, and let $g_1,\ldots,g_c\in\mfrak_{Y,\xi}$.
The following are equivalent:
\begin{enumerate}[label=\textup{(\roman*)}]
\item the classes of $g_1,\ldots,g_c$ in $\mfrak_{Y,\xi}/\mfrak_{Y,\xi}^2$ are linearly independent;
\item $g_1,\ldots,g_c$ extend to a regular system of parameters of $\sheaf O_{Y,\xi}$;
\item $V(g_1,\ldots,g_c)$ is regular of codimension $c$ at $\xi$.
\end{enumerate}
\end{lemma}

\section{Normal flatness and permissible stratifications}
\label{sec:permiss}

We recall some results of Hironaka from \cite{HironakaI,HironakaII} on normal flatness and permissible primes, though generally follow Bennett's notation from \cite{bennett}.

\begin{definition}[{\cite[Definition~1]{HironakaI}, \cite[Ch.\,0, \S\,2.1]{bennett}}]
\label{defi:normal-flatness}
Let $X$ be a Noetherian scheme, let $Y\subseteq X$ be a closed subscheme, and let $\xi\in Y$.
If $I_{Y,\xi}\subseteq\sheaf O_{X,\xi}$ is the defining ideal, then $X$ is \emph{normally flat along $Y$ at $\xi$} if $Y$ is regular at $\xi$ and $\gr_{I_{Y,\xi}}(\sheaf O_{X,\xi})$
is flat over $\sheaf O_{Y,\xi}$. (Equivalently, every $I^n/I^{n+1}$ is a free $R/I$-module; see
\cite[\href{https://stacks.math.columbia.edu/tag/00NZ}{Tag~00NZ}]{stacks-project}.)
\end{definition}

\begin{definition}[{\cite[Definition~8]{HironakaII}, \cite[Ch.\,0, \S\,2.1]{bennett}}]
\label{defi:permissible-prime}
Let $(R,\mfrak)$ be a Noetherian local ring.
A prime ideal $\pfrak\subseteq R$ is \emph{permissible} if $R/\pfrak$ is regular and $\gr_\pfrak(R)$ is flat over $R/\pfrak$.
\end{definition}

\begin{theorem}[{\cite[Proposition~1, Theorem~2]{HironakaI}, \cite[Ch.\,0, (2.1.1)]{bennett}}]
\label{thm:Bennett}
Let $(R,\mfrak,\kappa)$ be Noetherian local and let $\pfrak\subseteq R$ be prime with $R/\pfrak$ regular.
Set
\[
  A=\gr_\pfrak(R)/\mfrak\gr_\pfrak(R)
  \qquad
  \text{and}
  \qquad
  B=\gr_{\mfrak/\pfrak}(R/\pfrak).
\]
Then $\pfrak$ is permissible if and only if the quasi-canonical map
\[
  A\otimes_\kappa B\longrightarrow\gr_\mfrak(R)
\]
is an isomorphism.
If $s=\dim(R/\pfrak)$, then
\[
  B\cong\kappa[u_1,\ldots,u_s]
  \qquad
  \text{and}
  \qquad
  \deg u_j=1.
\]
\end{theorem}

\begin{theorem}[{\cite[Theorem~1]{HironakaI}, \cite[Ch.\,0, (2.2.1)]{bennett}}]
\label{thm:generic-nf}
Let $X$ be of finite type over a field and let $Z\subseteq X$ be an integral closed subscheme.
There is a nonempty open subset
\[
  U\subseteq\Reg(Z)
\]
such that $X$ is normally flat along $Z$ at every point of $U$.
\end{theorem}

\begin{proof}
Let $\mathcal I$ be the ideal of $Z$.
The normal cone
\[
 C_ZX=\Spec_Z\!\left(\bigoplus_{n\ge0}\mathcal I^n/\mathcal I^{n+1}\right)
\]
is of finite type over $Z$.
Generic flatness (\cite[\href{https://stacks.math.columbia.edu/tag/052A}{Tag~052A}]{stacks-project}) gives a dense open $U_0\subseteq Z$ over which $C_ZX\to Z$ is flat.
Since $Z$ is of finite type over a field, it is $J$-2 by
\cite[\href{https://stacks.math.columbia.edu/tag/07R5}{Tag~07R5}]{stacks-project}; hence $\Reg(Z)$ is open.
It contains the generic point of $Z$, so
$U=U_0\cap\Reg(Z)$ is nonempty.
For $x\in U$, the graded $\sheaf O_{Z,x}$-algebra
\[
 \bigoplus_{n\ge0}\mathcal I_x^n/\mathcal I_x^{n+1}
\]
is flat; hence $X$ is normally flat along $Z$ at $x$.
\end{proof}

\begin{definition}
\label{defi:perm-strat}
Let $X$ be of finite type over a field.
A \emph{regular permissible stratification} is a finite partition
\[
  X_{\red}=Y_1\sqcup\cdots\sqcup Y_N
\]
by regular irreducible locally closed subschemes $Y_i$ such that, for
\[
  Z_i=(\overline{Y_i})_{\red}\subseteq X_{\red},
\]
the scheme $X$ is normally flat along $Z_i$ at every point of $Y_i$.
The normal-flatness condition is imposed on the local rings of $X$, not on those of $X_{\red}$.

The stratification is \emph{smooth permissible} if each $Y_i$ is smooth over the ground field.
\end{definition}

\begin{theorem}
\label{thm:stratification}
Every finite-type scheme over a field admits a regular permissible stratification.
Over a perfect field, the stratification may be chosen smooth.
\end{theorem}

\begin{proof}
We argue by Noetherian induction on reduced closed subschemes $W\subseteq X_{\red}$.
The assertion is trivial for $W=\varnothing$.
Let $Z$ be an irreducible component of $W$, and write $Z_1,\ldots,Z_t$ for the other irreducible components.
The regular locus of $Z$ is a nonempty open subset by \cite[\href{https://stacks.math.columbia.edu/tag/07R5}{Tag~07R5}]{stacks-project}.
By \zcref{thm:generic-nf}, there is a nonempty open $U_0\subseteq\Reg(Z)$ along which $X$ is normally flat over $Z$.
Since
\[
 Z\not\subseteq Z_1\cup\cdots\cup Z_t,
\]
the open subset
\[
 U_1=U_0\setminus(Z_1\cup\cdots\cup Z_t)
\]
of $Z$ is nonempty.
Choose an open $O\subseteq W$ with $O\cap Z=U_1$ and set
\[
 U=O\setminus(Z_1\cup\cdots\cup Z_t).
\]
Every point of $W$ lies on an irreducible component; hence $U=U_1$.
Thus $U$ is nonempty, regular, irreducible, open in $W$, and $\overline{U}_{\red}=Z$.
It is permissible in $X$.

The reduced closed subscheme
\[
  W'=(W\setminus U)_{\red}
\]
is proper in $W$.
Apply the induction hypothesis to $W'$ and adjoin the stratum $U$.

Taking $W=X_{\red}$ proves the first assertion.
If the ground field is perfect, every regular finite-type scheme over it is smooth; see \cite[Tags~\href{https://stacks.math.columbia.edu/tag/038V}{038V} and \href{https://stacks.math.columbia.edu/tag/038X}{038X}]{stacks-project}.
\end{proof}

\begin{remark}
\label{rem:imperfect}
Over an arbitrary field, \zcref{thm:stratification} gives regular strata.
Smoothness is required for the finite-field sieves below and, over a possibly-imperfect infinite field, for \zcref{thm:infinite-field}.
\end{remark}

\begin{remark}[Simple normal crossings]
\label{rem:snc-strata}
Suppose that the completed local pair along an SNC stratum has the form
\[
 \widehat{\sheaf O}_{X,x}
 \cong
 \kappa(x)[[t_1,\ldots,t_r,u_1,\ldots,u_s]]/(t_1\cdots t_r),
\]
with stratum ideal $(t_1,\ldots,t_r)$.
Then
\begin{align*}
  \gr_{(t_1,\ldots,t_r)}(\widehat{\sheaf O}_{X,x})
  &\cong \kappa(x)[[u_1,\ldots,u_s]]
       [T_1,\ldots,T_r]/(T_1\cdots T_r).
\end{align*}
This ring is free over $\kappa(x)[[u_1,\ldots,u_s]]$.
Permissibility is invariant under completion by \cite[Ch.\,0, (2.2.2)]{bennett}; hence the usual SNC strata are permissible.
\end{remark}

\begin{remark}[Isolated singularities]
\label{ex:isolated}
Assume that $X$ is reduced and $\Sing(X)$ is a finite set of closed points.
For each irreducible component $Z$ of $X$, let
\[
 Y_Z=Z\cap\Reg(X).
\]
The nonempty $Y_Z$, together with the points of $\Sing(X)$, form a permissible stratification.
At a point of $Y_Z$, the local ideal of $Z$ is zero.
At a singular closed point, the stratum ideal is the maximal ideal, and its associated graded ring is flat over the residue field.
\end{remark}

\section{The bridge lemma}
\label{sec:bridge}

Let $\pfrak\subset(R,\mfrak,\kappa)$ be a permissible prime ideal.
The decomposition in \zcref{thm:Bennett} separates the tangent cone into the normal cone to $V(\pfrak)$ and the polynomial tangent directions along $V(\pfrak)$.
We record the following elementary observation.

\begin{lemma}
\label{lem:monic-linear-sequence}
Let $A$ be a $\kappa$-algebra, let $V$ be a finite-dimensional $\kappa$-vector space, and set
\[
  B=A\otimes_\kappa\Sym_\kappa(V).
\]
For $1\le j\le c$, let $h_j=a_j+\ell_j$, where $a_j\in A$ and $\ell_j\in V$.
If $\ell_1,\ldots,\ell_c$ are linearly independent, then $h_1,\ldots,h_c$ is a $B$-regular sequence.
If $\ell_1,\ldots,\ell_c,v_{c+1},\ldots,v_s$ is a basis of $V$, then
\[
  B/(h_1,\ldots,h_c)\cong A[v_{c+1},\ldots,v_s].
\]
\end{lemma}

\begin{proof}
Use the given basis.
Then
\[
  B=A[\ell_1,\ldots,\ell_c,v_{c+1},\ldots,v_s].
\]
Successively substitute $\ell_j=-a_j$.
At the $j$th step, $\ell_j+a_j$ is monic in $\ell_j$, hence a nonzerodivisor.
The final quotient is the polynomial ring on the right-hand side.
\end{proof}

\begin{lemma}
\label{lem:bennett-degree-one}
With the setup of \zcref{thm:Bennett}, assume that $\pfrak\subseteq R$ is a permissible prime, and set $G=\gr_\mfrak(R)$.
In degree one the quasi-canonical isomorphism is
\[
  G_1\cong A_1\oplus B_1.
\]
Under this identification, the map
\[
  \mfrak/\mfrak^2\longrightarrow
  (\mfrak/\pfrak)/(\mfrak/\pfrak)^2
\]
is the projection $A_1\oplus B_1\to B_1$.
\end{lemma}

\begin{proof}
Since $A_0=B_0=\kappa$,
\[
  (A\otimes_\kappa B)_1
  =(A_1\otimes_\kappa B_0)\oplus(A_0\otimes_\kappa B_1)
  =A_1\oplus B_1.
\]
The assertion about the quotient map is part of the construction of the quasi-canonical homomorphism; see \cite[Ch.\,0, (2.1.1)]{bennett}.
\end{proof}

\begin{theorem}[Bridge lemma]
\label{thm:bridge-ci}
Let $(R,\mathfrak m,\kappa)$ be a Noetherian local ring and let
$\mathfrak p\subseteq R$ be a permissible prime.
Let
$f_1,\ldots,f_c\in\mathfrak m$.
If their images in
\[
 \frac{\mathfrak m/\mathfrak p}{(\mathfrak m/\mathfrak p)^2}
\]
are linearly independent over $\kappa$, then
$f_1^*,\ldots,f_c^*$ is a regular sequence on
$\gr_\mathfrak m(R)$.
Consequently $f_1,\ldots,f_c$ is an $R$-regular sequence,
\begin{equation}
\label{eq:bridge-graded}
  \gr_\mfrak(R)/(f_1^*,\ldots,f_c^*)
  \xrightarrow{\ \sim\ }
  \gr_{\mfrak/(f_1,\ldots,f_c)}
  \bigl(R/(f_1,\ldots,f_c)\bigr),
\end{equation}
and
\begin{equation}
\label{eq:bridge-mult}
  e\bigl(R/(f_1,\ldots,f_c)\bigr)=e(R).
\end{equation}
\end{theorem}

\begin{proof}
By \zcref{thm:Bennett},
\[
  \gr_\mfrak(R)\cong A\otimes_\kappa B,
  \qquad
  B\cong\kappa[u_1,\ldots,u_s],
  \qquad
  \text{and}
  \qquad
  B_1\cong\nfrak/\nfrak^2.
\]
The hypothesis implies $f_j\notin\mfrak^2$.
By \zcref{lem:bennett-degree-one},
\[
  f_j^*=a_j+\ell_j\in A_1\oplus B_1,
\]
where $\ell_1,\ldots,\ell_c$ are linearly independent.
Apply \zcref{lem:monic-linear-sequence}; then apply \zcref{lem:initial-regular-sequence} to obtain \eqref{eq:bridge-graded} and \eqref{eq:bridge-mult}.
\end{proof}

\begin{remark}
\zcref{thm:bridge-ci} does not assume reducedness, equidimensionality, or excellence.
\end{remark}

\begin{corollary}
\label{thm:bridge}
Let $(R,\mfrak,\kappa)$ be a Noetherian local ring, let $\pfrak\subseteq R$ be permissible prime, and let $f\in\mfrak$.
If the image $\bar f\in R/\pfrak$ satisfies
\[
  \bar f\in(\mfrak/\pfrak)\setminus(\mfrak/\pfrak)^2,
\]
then $f\in\mfrak\setminus\mfrak^2$, the elements $f^*$ and $f$ are nonzerodivisors on $\gr_\mfrak(R)$ and $R$, respectively, and
\[
  e(R/(f))=e(R).
\]
\end{corollary}

\begin{proof}
Apply \zcref{thm:bridge-ci} with $c=1$.
\end{proof}

\begin{corollary}
\label{cor:transverse}
Let $X$ be of finite type over a field, let $X_{\red}=\bigsqcup_iY_i$ be a regular permissible stratification, and let $\xi\in Y_i$.
If $f_1,\ldots,f_c\in\mfrak_{X,\xi}$ have linearly independent images in
\[
 \mfrak_{Y_i,\xi}/\mfrak_{Y_i,\xi}^2,
\]
then $f_1,\ldots,f_c$ is $\sheaf O_{X,\xi}$-regular and
\[
 e\bigl(\sheaf O_{X,\xi}/(f_1,\ldots,f_c)\bigr)
 =e(\sheaf O_{X,\xi}).
\]
For $c=1$, one also has $\ord_\xi(f_1)=1$.
\end{corollary}

\begin{proof}
Let $Z_i=(\overline{Y_i})_{\red}$.
Its local ideal $\pfrak_{i,\xi}\subseteq\sheaf O_{X,\xi}$ is permissible and
\[
 \sheaf O_{X,\xi}/\pfrak_{i,\xi}\cong\sheaf O_{Y_i,\xi}.
\]
Now apply \zcref{thm:bridge-ci}.
\end{proof}

\begin{corollary}
\label{cor:transverse-sections-hilbert-samuel}
In the setting of \zcref{cor:transverse}, set
\[
 X'=V(f_1,\ldots,f_c)\subseteq\Spec\sheaf O_{X,\xi}.
\]
Then
\[
 \gr_{\mfrak_{X',\xi}}\sheaf O_{X',\xi}
 \cong
 \gr_{\mfrak_{X,\xi}}\sheaf O_{X,\xi}/(f_1^*,\ldots,f_c^*)
\].
Writing
\[
 \chi_{X,\xi}(n)=\ell\!\left(\sheaf O_{X,\xi}/\mfrak_{X,\xi}^{n+1}\right),
 \qquad
 \chi_{X',\xi}(n)=\ell\!\left(\sheaf O_{X',\xi}/\mfrak_{X',\xi}^{n+1}\right),
\]
one has
\[
 \chi_{X',\xi}(n)=\Delta^c\chi_{X,\xi}(n)
 \qquad n\ge0.
\]
In particular,
\[
 e_j(\sheaf O_{X',\xi})=e_j(\sheaf O_{X,\xi})
 \qquad
 0\le j\le\dim\sheaf O_{X,\xi}-c.
\]
\end{corollary}

\begin{proof}
By \zcref{thm:bridge-ci}, the initial forms $f_1^*,\ldots,f_c^*$ form a regular sequence on
$\gr_{\mfrak_{X,\xi}}\sheaf O_{X,\xi}$.
Apply \zcref{lem:initial-regular-sequence,cor:hilbert-samuel-finite-differences}.
\end{proof}

\section{Proof of the main theorem}
\label{sec:main-proof}

Now we prove the main result of this paper: a Bertini theorem for Hilbert--Samuel multiplicity over finite fields.

\begin{proof}[Proof of \zcref{thm:intro-main}]
For a closed point $\xi\in Y_i$, put $Q_\xi=q^{\deg \xi}$.
The intersection $Y_i\cap H_f$ is smooth of codimension one at $\xi$ if and only if the image of $f$ in
\[
 \sheaf O_{Y_i,\xi}/\mfrak_{Y_i,\xi}^2
\]
is nonzero.
Since $Y_i$ is smooth of dimension $s_i$,
\[
  \dim_{\kappa(\xi)}\sheaf O_{Y_i,\xi}/\mfrak_{Y_i,\xi}^2=s_i+1;
\]
hence the local good measure is
\[
  1-Q_\xi^{-(s_i+1)}.
\]
Apply \cite[Theorem~1.3 and the subsequent remark]{poonen:bertini} to these local conditions and to the finite family $Y_1,\ldots,Y_N$.
This gives
\[
  \mu(\cP)
  =\prod_{i=1}^N\prod_{\xi\in|Y_i|}
    \bigl(1-q^{-(s_i+1)\deg \xi}\bigr)
  =\prod_{i=1}^N\zeta_{Y_i}(s_i+1)^{-1}.
\]
Every local factor is nonzero, so the product is positive by the remark following \cite[Theorem~1.3]{poonen:bertini}.

Let $f\in\cP$.
At every closed point of $Y_i\cap H_f$, \zcref{lem:smooth-first-jets} gives a nonzero first jet.
The remaining assertions at every point of $X\cap H_f$ follow from \zcref{prop:closed-to-all-multiplicity} with $c=1$.
\end{proof}

\section{Examples and first applications}
\label{sec:examples}

We collect a few concrete applications of \zcref{thm:intro-main}.

\subsection{Reduced curves with finitely many singular points}

Let $C\subseteq\bP^n_{\Fq}$ be a reduced curve with singular closed points $\xi_1,\ldots,\xi_t$.
The regular loci of the irreducible components, together with the points $\xi_j$, form a permissible stratification.
Since their union is $\Reg(C)$, \zcref{thm:intro-main} gives the density
\[
  \zeta_{\Reg(C)}(2)^{-1}
  \prod_{j=1}^t(1-q^{-\deg \xi_j}).
\]
For every hypersurface in this set, multiplicity is preserved at every point of $C\cap H_f$.

\subsection{Isolated singularities}

Let $X$ be reduced and pure of dimension $m$, and assume that $\Sing(X)$ is finite.
The stratification in \zcref{ex:isolated} gives
\[
  \zeta_{\Reg(X)}(m+1)^{-1}
  \prod_{\xi\in\Sing(X)}(1-q^{-\deg \xi}).
\]

\subsection{Simple normal crossings}

Let $X$ be a split simple normal crossings scheme over $\Fq$, with its standard strata.
A stratum $Y$ of dimension $s$ contributes
\[
  \zeta_Y(s+1)^{-1}.
\]
Thus the density is the product of these factors over the SNC strata.

\subsection{Zero-dimensional strata}

If $Y=\{\xi\}$ is a zero-dimensional stratum, the required intersection has dimension $-1$; equivalently, $\xi\notin H_f$.
Its local factor is
\[
  1-q^{-\deg \xi}.
\]

\section{Complete intersections}
\label{sec:bk}

For a prime power $q$ and integers $s,c\ge0$, set
\[
  L(q,s,c)=
  \begin{cases}
    \displaystyle\prod_{j=0}^{c-1}(1-q^{j-s}), & c\le s,\\[6pt]
    0, & c>s.
  \end{cases}
\]
If $c\le s$, this is the proportion of linearly independent $c$-tuples in an $s$-dimensional vector space over $\mathbb F_q$.

\begin{theorem}
\label{thm:bk-mult}
Let $X\subseteq\bP^n_{\Fq}$ be quasiprojective, and let
\[
 X_{\red}=Y_1\sqcup\cdots\sqcup Y_N
\]
be a smooth permissible stratification.
Put $s_i=\dim Y_i$, fix $c\ge1$, and set
\[
 k_i=\min\{c,s_i+1\}.
\]
For $d_1\le\cdots\le d_c$, let $\cP_{\mathbf d}$ be the set of tuples
\[
 \mathbf f=(f_1,\ldots,f_c)\in S_{d_1}\times\cdots\times S_{d_c}
\]
such that, for every $i$,
\[
 Y_i\cap H_{f_1}\cap\cdots\cap H_{f_c}
\]
is smooth of dimension $s_i-c$ if $c\le s_i$, and is empty if $c>s_i$.
Assume that $d_1\to\infty$ and, for every $i$ with $s_i>0$,
\begin{equation}
\label{eq:bk-growth}
 d_{k_i}^{s_i}q^{-d_1/\max\{s_i+1,p\}}\longrightarrow 0,
 \qquad p=\operatorname{char}\Fq.
\end{equation}
Then
\begin{equation}
\label{eq:bk-density}
 \frac{\#\cP_{\mathbf d}}
 {\#(S_{d_1}\times\cdots\times S_{d_c})}
 \longrightarrow
 \prod_{i=1}^N\prod_{P\in|Y_i|}\alpha_{i,P},
\end{equation}
where, for $Q_P=q^{\deg P}$,
\[
 \alpha_{i,P}
 =1-Q_P^{-c}+Q_P^{-c}L(Q_P,s_i,c).
\]
The product is positive.
If $\mathbf f\in\cP_{\mathbf d}$ and
\[
 \xi\in X\cap H_{f_1}\cap\cdots\cap H_{f_c},
\]
then the local equations form an $\sheaf O_{X,\xi}$-regular sequence and
\[
 e_\xi(X\cap H_{f_1}\cap\cdots\cap H_{f_c})=e_\xi(X).
\]
\end{theorem}

\begin{proof}
Fix $P\in Y_i$ and put $Q=q^{\deg P}$.
Trivialize each $\sheaf O(d_j)$ on the first infinitesimal neighborhood of $P$ in $Y_i$.
If the constant terms of $f_1,\ldots,f_c$ are not all zero, then $P$ is not in the common zero scheme.
All constant terms vanish with probability $Q^{-c}$; conditional on this event, the differentials are uniformly distributed in
\[
 \bigl(\mfrak_{Y_i,P}/\mfrak_{Y_i,P}^2\bigr)^c.
\]
Their rank is $c$ with probability $L(Q,s_i,c)$.
Thus the local good factor is
\[
 \alpha_{i,P}=1-Q^{-c}+Q^{-c}L(Q,s_i,c).
\]

Fix $r>0$.
Let $W_r$ be the union, in $\bP^n_{\Fq}$, of the first infinitesimal neighborhoods in the strata of all closed points of degree $<r$.
Each component is an immersion with closed support, hence a closed immersion by
\cite[\href{https://stacks.math.columbia.edu/tag/01IQ}{Tag~01IQ}]{stacks-project}; the supports are distinct, so $W_r$ is finite.
For $d_1\gg0$, with the bound depending on $r$, all maps
\[
 S_{d_j}\longrightarrow H^0(W_r,\sheaf O_{W_r}(d_j))
\]
are surjective by \cite[Lemma~2.1]{bucur-kedlaya}.
The conditions at points of degree $<r$ are independent and have density
\[
 \prod_{i=1}^N
 \prod_{\substack{P\in|Y_i|\\ \deg P<r}}
 \alpha_{i,P}.
\]

Suppose $s_i>0$.
If $c\le s_i$, apply
\cite[Lemma~2.5 and Corollary~2.7]{bucur-kedlaya} to all $c=k_i$ equations.
If $c>s_i$, a common zero of all $c$ equations is a common zero of the first
$k_i=s_i+1$ equations, and every common zero of those $k_i$ equations is bad because $k_i>s_i$.
The same estimates give
\begin{equation}
\label{eq:bk-tail}
\begin{aligned}
 &\Pr\bigl(\text{a bad point of }Y_i\text{ has degree at least }r\bigr)\\
 &\qquad\le C_iq^{-r(2k_i-1)}
 +C_i'd_{k_i}^{s_i}q^{-d_1/\max\{s_i+1,p\}}.
\end{aligned}
\end{equation}
with constants independent of $r$ and $\mathbf d$.
The second term tends to zero by \eqref{eq:bk-growth}.
A zero-dimensional stratum is finite and is contained in $W_r$ for $r\gg0$.
Summing \eqref{eq:bk-tail} over $i$, then letting $\mathbf d\to\infty$ and $r\to\infty$, proves \eqref{eq:bk-density}.

For $c\le s_i$,
\begin{align*}
 1-\alpha_{i,P}
 &=Q^{-c}\bigl(1-L(Q,s_i,c)\bigr)\\
 &\le Q^{-c}\sum_{j=0}^{c-1}Q^{j-s_i}
 \le cQ^{-s_i-1}.
\end{align*}
For $c>s_i$, one has $1-\alpha_{i,P}=Q^{-c}\le Q^{-s_i-1}$.
Hence
\[
 \sum_{P\in|Y_i|}(1-\alpha_{i,P})<\infty
\]
by \cite[Remark after Theorem~1.3]{poonen:bertini}.
Every factor is positive.
Since
$-\log x\le2(1-x)$ for $1/2\le x\le1$, the Euler product is positive.

Let $\xi$ belong to the common zero scheme and choose $i$ with $\xi\in Y_i$.
Then $c\le s_i$, and the first jets have rank $c$ at every closed point of the intersection.
Now apply \zcref{prop:closed-to-all-multiplicity}.
\end{proof}

\begin{remark}
The single condition
\[
 d_c^s q^{-d_1/\max\{s+1,p\}}\longrightarrow0,
 \qquad s=\max_i s_i,
\]
is sufficient for \eqref{eq:bk-growth}, but is not needed.
\end{remark}

For $c=1$, \eqref{eq:bk-density} is \eqref{eq:intro-density}.

\section{Hypersurfaces containing a prescribed subscheme}
\label{sec:gw}
Let $Z\subseteq\bP^n_{\Fq}$ be a closed subscheme with ideal sheaf $\mathcal I_Z$ and set
\[
  I_d=H^0(\bP^n_{\Fq},\mathcal I_Z(d))
  \qquad
  \text{and}
  \qquad
  I_{\mathrm{homog}}=\bigcup_{d\ge0}I_d.
\]
For $\cP\subseteq I_{\mathrm{homog}}$, its relative density is
\[
  \mu_Z(\cP)=\lim_{d\to\infty}\frac{\#(\cP\cap I_d)}{\#I_d},
\]
when the limit exists.

If $Y$ is smooth and $V=Z\cap Y$, let $V_e$ be the rank-$e$ stratum of $\Omega^1_{V/\Fq}$, as in \cite[Theorem~1.2]{wutz2016}.
Its closed points are precisely the points $\xi\in V$ with
\[
  \edim_\xi(V)
  =\dim_{\kappa(\xi)}\mfrak_{V,\xi}/\mfrak_{V,\xi}^2=e.
\]

\begin{theorem}
\label{thm:gw-mult}
Let $X\subseteq\bP^n_{\Fq}$ be quasiprojective, let
\[
  X_{\red}=Y_1\sqcup\cdots\sqcup Y_N
\]
be a smooth permissible stratification, and put
\[
  s_i=\dim Y_i,
  \qquad
  V_i=Z\cap Y_i,
  \qquad
  \text{and}
  \qquad
  V_{i,e}=(V_i)_e.
\]
Assume
\begin{equation}
\label{eq:gw-admissible}
  e+\dim V_{i,e}<s_i
\end{equation}
whenever $V_{i,e}\ne\varnothing$.
Let $\cP_Z$ be the set of $f\in I_{\mathrm{homog}}$ such that $Y_i\cap H_f$ is smooth of dimension $s_i-1$ for every $i$, with dimension $-1$ meaning the empty scheme.
Then
\begin{equation}
\label{eq:gw-density}
  \mu_Z(\cP_Z)
  =\prod_{i=1}^N
    \left(
      \zeta_{Y_i\setminus V_i}(s_i+1)^{-1}
      \prod_{e=0}^{s_i-1}\zeta_{V_{i,e}}(s_i-e)^{-1}
    \right)>0.
\end{equation}
For $f\in\cP_Z$ and $\xi\in X\cap H_f$,
\begin{gather*}
  \ord_\xi(f)=1,
  \qquad
  f_\xi\text{ is a nonzerodivisor on }\sheaf O_{X,\xi},
  \\
  \text{and}
  \quad
  e_\xi(X\cap H_f)=e_\xi(X).
\end{gather*}
\end{theorem}

\begin{proof}
Let $P\in V_{i,e}$, put $(A,\mfrak)=\sheaf O_{Y_i,P}$, and let $J\subseteq A$ define $V_i$ at $P$.
Since $Y_i$ is smooth of dimension $s_i$,
\[
 \dim_{\kappa(P)}\mfrak/\mfrak^2=s_i,
 \qquad
 \dim_{\kappa(P)}\frac{\mfrak/J}{(\mfrak/J)^2}=e.
\]
The exact sequence
\[
 0\longrightarrow\frac{J+\mfrak^2}{\mfrak^2}
 \longrightarrow\frac{\mfrak}{\mfrak^2}
 \longrightarrow\frac{\mfrak/J}{(\mfrak/J)^2}
 \longrightarrow0
\]
gives
\[
 \dim_{\kappa(P)}\frac{J+\mfrak^2}{\mfrak^2}=s_i-e.
\]
A section in $I_d$ has first jet in this kernel, and it is singular on $Y_i$ at $P$ precisely when that first jet is zero.
Thus the bad condition has codimension $(s_i-e)\deg P$, as in \cite[Lemma~2.3]{wutz2016}.
If $P\in Y_i\setminus V_i$, the full first-jet space
$\sheaf O_{Y_i,P}/\mfrak_{Y_i,P}^2$ has dimension $s_i+1$ over $\kappa(P)$, so the bad condition has codimension $(s_i+1)\deg P$.
The local factors are
\[
 1-q^{-(s_i+1)\deg \xi},
 \qquad
 1-q^{-(s_i-e)\deg \xi}.
\]

Apply \cite[Proposition~4.8]{ghosh-krishna} to the pairwise disjoint smooth strata $Y_i$, with the finite local-condition scheme empty.
Its numerical hypothesis is \eqref{eq:gw-admissible}, and its Euler product is \eqref{eq:gw-density}.

The first jet is nonzero at every closed point of $Y_i\cap H_f$.
Apply \zcref{prop:closed-to-all-multiplicity} with $c=1$.
\end{proof}
\begin{remark}
Let $F\subseteq\bP^n_{\Fq}$ be finite and disjoint from $Z$.
Since $\mathcal I_Z\sheaf O_F=\sheaf O_F$, choose on each connected component $F_a$ the least $j(a)$ for which $x_{j(a)}$ is invertible and define
\[
 f|_F=(x_{j(a)}^{-d}f|_{F_a})_a\in H^0(F,\sheaf O_F)
 \qquad(f\in I_d).
\]
Let $\varnothing\ne\mathcal T\subseteq H^0(F,\sheaf O_F)$, and put
\[
 U_i=Y_i\setminus F,
 \qquad
 V_i^\circ=Z\cap U_i,
 \qquad
 \text{and}
 \qquad
 V_{i,e}^\circ=(V_i^\circ)_e.
\]
Then \cite[Proposition~4.8]{ghosh-krishna} gives
\begin{align*}
 &\mu_Z\bigl(\{f:
 f|_F\in\mathcal T,
 H_f\cap U_i\text{ is smooth of dimension }s_i-1
 \text{ for all }i\}\bigr)\\
 &\quad=
 \frac{\#\mathcal T}{\#H^0(F,\sheaf O_F)}
 \prod_{i=1}^N
 \biggl(
  \zeta_{U_i\setminus V_i^\circ}(s_i+1)^{-1}
  \prod_{e=0}^{s_i-1}
  \zeta_{V_{i,e}^\circ}(s_i-e)^{-1}
 \biggr).
\end{align*}
The finite factor occurs once; no Euler factor supported on $F\cap X_{\red}$ occurs in the product.

More general first-order Taylor conditions, including finite conditions and conditions defined by locally free quotients of first differentials, are treated in \cite[Theorems~A--C]{Matthew}.
The general higher-order case will be treated in an upcoming preprint by the first-named author.

Suppose that $F$ contains the first infinitesimal neighborhood $\xi^{(2)}_{Y_i}$ for every $\xi\in F\cap Y_i$, and that the image of every $t\in\mathcal T$ in
\[
 H^0(\xi^{(2)}_{Y_i},\sheaf O_{\xi^{(2)}_{Y_i}})
\]
is nonzero.
Then the counted sections are transverse at these points, and the conclusions of \zcref{thm:gw-mult} hold there as well.
\end{remark}

\section{Semiample linear systems}
\label{sec:semiample}

Let $X$ be projective over $\Fq$, not necessarily reduced.
Let $A$ be a very ample Cartier divisor on $X$, let $E$ be a globally generated Cartier divisor, and let
\[
  \pi:X\longrightarrow \bP^M_{\Fq}
\]
be the morphism defined by $|E|$.
Put
\[
  \sheaf L_{n,d}=\sheaf O_X(nA+dE)
  \qquad
  \text{and}
  \qquad
  R_{n,d}=H^0(X,\sheaf L_{n,d}).
\]
For $s\in R_{n,d}$, let $D_s\subseteq X$ be its zero scheme.
If $\xi\in D_s$, choose a trivialization of $\sheaf L_{n,d}$ at $\xi$ and let $f_{s,\xi}\in\sheaf O_{X,\xi}$ be the corresponding local equation.
The principal ideal $(f_{s,\xi})$ is independent of the trivialization.

Fix a smooth permissible stratification
\[
  X_{\red}=Y_1\sqcup\cdots\sqcup Y_N,
  \qquad s_i=\dim Y_i,
\]
and put $C_i=(\overline{Y_i})_{\red}\subseteq X_{\red}$.
Let $B=\pi(X_{\red})_{\red}\subseteq\bP^M_{\Fq}$ with its reduced induced structure.
Let $\rho:X_{\red}\to B$ be the induced morphism and put
\[
 \mathcal F_n=\rho_*\sheaf O_{X_{\red}}(nA).
\]
Here and below, ``smoothly in codimension one'' on $Y_i$ means smooth of dimension $s_i-1$; when $s_i=0$, it means empty.
For $P\in |B|$, set
\[
  P^{(2)}=\Spec(\sheaf O_{B,P}/\mfrak_P^2),
  \qquad
  X_{P^{(2)}}=X_{\red}\times_B P^{(2)}.
\]
Since $P^{(2)}$ is local Artinian, $\sheaf O_B(d)|_{P^{(2)}}$ is trivial.

For fixed $n$ and $P\in |B|$, define
\begin{equation}
\label{eq:stable-image}
 I_P(n)=\operatorname{im}\!\left(
 H^0(P^{(2)},\mathcal F_n|_{P^{(2)}})
 \longrightarrow
 H^0(X_{P^{(2)}},\sheaf O_{X_{\red}}(nA))
 \right).
\end{equation}
For $Q\in Y_i\cap\rho^{-1}(P)$, put
\[
 Q^{(2)}_{Y_i}=\Spec(\sheaf O_{Y_i,Q}/\mfrak_{Y_i,Q}^2).
\]
For $u\in I_P(n)$, let $u_{i,Q}$ denote its image in
\[
 H^0(Q^{(2)}_{Y_i},\sheaf O_{X_{\red}}(nA)|_{Q^{(2)}_{Y_i}}).
\]
Set
\[
 U_P(n)=\{u\in I_P(n):u_{i,Q}\ne0
 \text{ for every }i\text{ and every }Q\in|Y_i\cap\rho^{-1}(P)|\},
\]
and
\begin{equation}
\label{eq:semiample-local-factor}
 \beta_P(n)=\frac{\#U_P(n)}{\#I_P(n)}.
\end{equation}
For $d\gg0$, $\beta_P(n)$ is the proportion of restrictions over $P$ that are transverse to every stratum.

\begin{lemma}
\label{lem:semiample-stable-image}
Let $T\subset|B|$ be finite and put $W_T=\coprod_{P\in T}P^{(2)}$.
After trivializing $\sheaf O_B(d)$ on each component of $W_T$, the image of
\[
 H^0(X_{\red},\sheaf O_{X_{\red}}(nA+dE))
 \longrightarrow
 H^0(X_{\red}\times_BW_T,\sheaf O_{X_{\red}}(nA))
\]
is, for $d\gg0$,
\[
 \bigoplus_{P\in T}I_P(n).
\]
\end{lemma}

\begin{proof}
By the projection formula, the source is $H^0(B,\mathcal F_n(d))$.
Serre vanishing gives a surjection
\[
 H^0(B,\mathcal F_n(d))
 \longrightarrow
 H^0(W_T,\mathcal F_n(d)|_{W_T})
\]
for $d\gg0$; see \cite[Ch.\,III, Theorem~5.2]{hartshorne}.
The target and the base-change map split over the connected components of $W_T$.
The image is therefore the direct sum above; compare with \cite[Lemma~5.2]{erman-wood}.
\end{proof}

\begin{lemma}
\label{lem:semiample-restriction-closed}
Let $W\subseteq X$ be closed.
There is $n_W$ such that, for $n\ge n_W$ and $d\gg0$, the restriction map
\[
  H^0(X,\sheaf O_X(nA+dE))
  \longrightarrow
  H^0(W,\sheaf O_W(nA+dE))
\]
is surjective.
\end{lemma}

\begin{proof}
Let $\sheaf I_W$ be the ideal of $W$.
Relative Serre vanishing gives, for $n\gg0$ and $j>0$,
\[
  R^j\pi_*(\sheaf I_W(nA))=0
\]
by \cite[Ch.\,III, Theorem~8.8]{hartshorne}.
Hence
\[
  0\longrightarrow\pi_*\sheaf I_W(nA)
  \longrightarrow\pi_*\sheaf O_X(nA)
  \longrightarrow\pi_*\sheaf O_W(nA)
  \longrightarrow0
\]
is exact.
Tensor with $\sheaf O_{\bP^M}(d)$, use the projection formula, and apply \cite[Ch.\,III, Theorem~5.2]{hartshorne}.
\end{proof}

\begin{lemma}
\label{lem:semiample-subsystem}
Let $C$ be an integral projective $\Fq$-scheme of dimension $m$, let $A_C$ be very ample, and let $E_C$ be generated by a finite-dimensional subspace
\[
 V\subseteq H^0(C,E_C).
\]
Let $\rho_V:C\to\bP(V^\vee)$ be the induced morphism, let $C^\circ\subseteq C$ be smooth and open, and put $b=\dim\rho_V(C)$.
For
\[
 n\ge\max\{b(m+1)-1,bp+1\},
\]
\cite[Theorem~9.10]{erman-wood} and the estimates in
\cite[Lemmas~3.2--3.4, 5.1--5.4, and~6.1, Propositions~7.1 and~7.2]{erman-wood}
remain valid for $H^0(C,nA_C+dE_C)$, with singular points restricted to $C^\circ$ and fibers taken with respect to $\rho_V$.
\end{lemma}

\begin{proof}
Let $\iota:C\hookrightarrow\bP^N$ be the embedding defined by $A_C$.
After choosing a basis of $V$, the graph map
\[
 j=(\iota,\rho_V):C\longrightarrow\bP^N\times\bP(V^\vee)
\]
is a closed immersion by
\cite[\href{https://stacks.math.columbia.edu/tag/01KS}{Tag~01KS}]{stacks-project}, and
\[
 j^*\sheaf O(1,0)=A_C
 \qquad
 \text{and}
 \qquad
 j^*\sheaf O(0,1)=E_C.
\]
Put
\[
 S_{a,d}=H^0\bigl(\bP^N\times\bP(V^\vee),\sheaf O(a,d)\bigr)
 \qquad
 \text{and}
 \qquad
 R_{a,d}=H^0(C,aA_C+dE_C).
\]
Restriction along $j$ gives $S_{a,d}\to R_{a,d}$.
The proofs in
\cite[\S\,5--7]{erman-wood} use the second projective factor through
this restriction map and the morphism induced by its homogeneous
coordinates.
With $j(C)$ in place of the embedding used there:
\begin{enumerate}[label=\textup{(\roman*)}]
\item \cite[Lemma~5.1]{erman-wood} bounds the sum of the $A_C$-degrees of the components of a fiber of $\rho_V$ by a constant times the degree of its image point;
\item the projection formula for $(\rho_V)_*\sheaf O_C(aA_C)$ and Serre vanishing give \cite[Lemma~5.2]{erman-wood};
\item the bihomogeneous-coordinate calculation in \cite[Lemma~5.3]{erman-wood} applies to $S_{a,d}\to R_{a,d}$;
\item in \cite[Lemma~5.4]{erman-wood}, the initial summand is chosen uniformly in $R_{n,d}$, so adding the auxiliary $p$-th power terms preserves the uniform distribution.
The bounds decoupling derivatives therefore apply with the coordinates of $\bP(V^\vee)$.
\end{enumerate}
These are the only parts from the complete linear series of $E_C$ in
\cite[Lemmas~3.2--3.4 and~6.1, Propositions~7.1 and~7.2]{erman-wood} that we are picking up.
Replacing
that series by $V$ changes neither their proofs nor the bound
$n\ge\max\{b(m+1)-1,bp+1\}$.
\end{proof}

Put
\[
 b_i=\dim\rho(C_i).
\]
Choose $\nu_i$ such that \zcref{lem:semiample-restriction-closed}, applied on $X_{\red}$ with $W=C_i$, holds for $n\ge\nu_i$, and set
\begin{equation}
\label{eq:semiample-n0}
 n_0=\max_i\{\nu_i,\ b_i(s_i+1)-1,\ b_ip+1,\ 1\}.
\end{equation}

\begin{proposition}
\label{prop:stratified-semiample-sieve}
Let $n\ge n_0$.
As $d\to\infty$, the density of sections of
\[
 H^0(X_{\red},\sheaf O_{X_{\red}}(nA+dE))
\]
whose zero scheme meets every $Y_i$ smoothly in codimension one is
\begin{equation}
\label{eq:semiample-product}
 \prod_{P\in|B|}\beta_P(n).
\end{equation}
The product converges, is zero only if one of its factors is zero, and is nonzero for all sufficiently large $n$.
\end{proposition}

\begin{proof}
For finite $T\subset|B|$, \zcref{lem:semiample-stable-image} gives the limiting probability
\[
 \prod_{P\in T}\beta_P(n)
\]
for transversality over $T$.

Fix $i$ and set
\[
 V_i=\operatorname{im}\!\left(
 H^0(X,E)\longrightarrow H^0(C_i,E|_{C_i})
 \right).
\]
The space $V_i$ generates $E|_{C_i}$ and defines $\rho|_{C_i}$.
For $n\ge n_0$ and $d\gg0$, restriction from $X_{\red}$ to $C_i$ is surjective by \zcref{lem:semiample-restriction-closed}; a uniform global section therefore restricts uniformly.
Apply \zcref{lem:semiample-subsystem} to
\[
 (C_i,Y_i,A|_{C_i},E|_{C_i},V_i).
\]
The medium- and high-degree estimates of \cite[Lemmas~3.3 and~3.4]{erman-wood} yield
\[
 \lim_{r\to\infty}\limsup_{d\to\infty}
 \Pr\!\left(
 D_s\cap Y_i\text{ is singular above some }P,
 \ \deg P\ge r
 \right)=0.
\]
Summing over the finitely many strata and then letting $r\to\infty$ proves \eqref{eq:semiample-product}.

Let $\sigma_{i,P}(n)$ be the limiting probability of a singular point on $Y_i$ above $P$.
Then
\[
 1-\beta_P(n)\le\sum_i\sigma_{i,P}(n).
\]
By \zcref{lem:semiample-subsystem} and \cite[Proposition~7.1]{erman-wood},
\[
 \sum_{P\in|B|}\sigma_{i,P}(n)<\infty.
\]
Hence $\sum_P(1-\beta_P(n))<\infty$.
If no factor vanishes, then $\beta_P(n)\ge\tfrac12$ for all but finitely many $P$, and
\[
 1-\beta_P(n)\le-\log\beta_P(n)
 \le2\bigl(1-\beta_P(n)\bigr).
\]
Thus the product is positive.
If $\beta_P(n)=0$ for some $P$, every partial product containing that factor is zero.
\end{proof}

\begin{lemma}
\label{lem:uniform-fiber-points}
For each $i$ there is a constant $c_i>0$ such that, for every $P\in|B|$ of degree $e$ and every $a\ge1$,
\begin{equation}
\label{eq:uniform-fiber-count}
 \#\bigl((Y_i\cap\rho^{-1}(P))(\mathbb F_{q^{ea}})\bigr)
 \le c_i e q^{ea s_i}.
\end{equation}
\end{lemma}

\begin{proof}
By \zcref{lem:semiample-subsystem} and \cite[Lemma~5.1]{erman-wood}, there is a constant $c_i'>0$, independent of $P$, such that the sum of the $A$-degrees of the geometric irreducible components of
$C_i\cap\rho^{-1}(P)$ is at most $c_i'e$.
After base change to $\kappa(P)=\mathbb F_{q^e}$, apply
\cite[Lemma~2.4]{bucur-kedlaya} to each irreducible component.
Since every component has dimension at most $s_i$ and the sum of their degrees is at most $c_i'e$,
\[
 \#\bigl((C_i\cap\rho^{-1}(P))(\mathbb F_{q^{ea}})\bigr)
 \le 2^{s_i}c_i'e q^{ea s_i}.
\]
The assertion follows after enlarging the constant.
\end{proof}

Let $\rho_i:C_i\to B$ be the restriction of $\rho$.
For $P\in|B|$, put
\[
 C_{i,P^{(2)}}=C_i\times_BP^{(2)}
\]
and define
\begin{equation}
\label{eq:stable-image-stratum}
 I_{i,P}(n)=\operatorname{im}\!\left(
 H^0\!\left(P^{(2)},(\rho_i)_*\sheaf O_{C_i}(nA)|_{P^{(2)}}\right)
 \longrightarrow
 H^0\!\left(C_{i,P^{(2)}},\sheaf O_{C_i}(nA)\right)
 \right).
\end{equation}

\begin{lemma}
\label{lem:semiample-stratum-image}
If $n\ge n_0$, restriction induces a surjection
\[
 I_P(n)\longrightarrow I_{i,P}(n)
\]
for every $i$ and every $P\in|B|$.
\end{lemma}

\begin{proof}
By the proof of \zcref{lem:semiample-restriction-closed}, the choice $n\ge n_0\ge\nu_i$ gives a surjection
\[
 \rho_*\sheaf O_{X_{\red}}(nA)
 \longrightarrow
 (\rho_i)_*\sheaf O_{C_i}(nA).
\]
Restrict to the affine scheme $P^{(2)}$ and take global sections.
The resulting surjection is compatible with the two base-change maps to
$X_{P^{(2)}}$ and $C_{i,P^{(2)}}$; hence it induces a surjection of their images
\[
 I_P(n)\longrightarrow I_{i,P}(n).
\]
\end{proof}

\begin{lemma}
\label{lem:semiample-positive-factors}
There is an integer $n_1\ge n_0$ such that
\[
 \beta_P(n)>0
\]
for every $n\ge n_1$ and every $P\in|B|$.
\end{lemma}

\begin{proof}
Fix $P\in|B|$ and put $e=\deg P$.
For $n\ge n_0$, \zcref{lem:semiample-stratum-image} shows that a uniform element of $I_P(n)$ maps uniformly to $I_{i,P}(n)$.
Thus the estimates for $C_i$ given by \zcref{lem:semiample-subsystem} apply to the corresponding events in $I_P(n)$.

All constants below are independent of $P$ and, after taking a maximum, of $i$.
Set
\[
 \gamma_{i,P}=
 \prod_{Q\in|Y_i\cap\rho^{-1}(P)|}
 \bigl(1-q^{-(s_i+1)\deg Q}\bigr).
\]
Every closed point in the fiber has degree $ea$ for some $a\ge1$.
Since $-\log(1-x)\le2x$ for $0\le x\le\tfrac12$, \zcref{lem:uniform-fiber-points} gives
\begin{align*}
 -\log\gamma_{i,P}
 &\le2\sum_{a\ge1}
 \#\bigl((Y_i\cap\rho^{-1}(P))(\mathbb F_{q^{ea}})\bigr)
 q^{-ea(s_i+1)}\\
 &\le2c_i e\sum_{a\ge1}q^{-ea}
 =\frac{2c_i e}{q^e-1}
 \le\frac{2c_i}{q-1}.
\end{align*}
Thus
\begin{equation}
\label{eq:uniform-low-product}
 \prod_{i=1}^N\gamma_{i,P}
 \ge
 \delta:=\exp\!\left(-\frac{2}{q-1}\sum_{i=1}^Nc_i\right)>0
\end{equation}
uniformly in $P$.

Choose $r$ larger than the degrees of all points on the zero-dimensional strata.
For $Q\in Y_i\cap\rho^{-1}(P)$, write $\deg Q=ea$.
In the range
\[
 ea>r,
 \qquad
 a\le\frac{n+1}{s_i+1},
\]
the proof of \cite[Lemma~6.1]{erman-wood} (in the fixed-generating-system form of \zcref{lem:semiample-subsystem}) gives singularity probability $q^{-ea(s_i+1)}$ at $Q$.
For fixed $a$, the number of closed points $Q$ of degree $ea$ is at most the left side of \eqref{eq:uniform-fiber-count}.
Hence the total medium-degree probability is at most
\begin{align*}
 c_i e\sum_{ea>r}q^{-ea}
 &=c_i\frac{e q^{-e(\lfloor r/e\rfloor+1)}}{1-q^{-e}}\\
 &\le c_i\frac{(r+1)q^{-(r+1)}}{1-q^{-1}}.
\end{align*}
For the last inequality, use $e\le r+1$ when $e\le r$; when $e>r$, the function $a\mapsto a/(q^a-1)$ is decreasing and $e\ge r+1$.
Choose $r$ so that the sum of these bounds over $i$ is less than $\delta/4$.

For $s_i>0$, put
\[
 J_i(n)=\left\lfloor\frac{n+1}{s_i+1}\right\rfloor+1
 \qquad
 \text{and}
 \qquad
 h_i(n)=\min\!\left\{J_i(n),
 \left\lfloor\frac{n-1}{p}\right\rfloor+1\right\}.
\]
With $J=J_i(n)$, \zcref{lem:semiample-subsystem} and \cite[Lemma~5.4(3)]{erman-wood} give
\[
 O\!\left((n^{s_i}+d^{s_i})q^{-(\lfloor d/p\rfloor+1)}
 +e(n^{s_i}+e^{s_i})q^{-eh_i(n)}\right).
\]
The first term tends to zero as $d\to\infty$.
Thus the limiting high-relative-degree probability is
\[
 O\!\left(e(n^{s_i}+e^{s_i})q^{-eh_i(n)}\right).
\]
For $h=h_i(n)$,
\begin{align*}
 \sup_{e\ge1}e(n^{s_i}+e^{s_i})q^{-eh}
 &\le n^{s_i}\sum_{a\ge1}a q^{-ah}
 +\sum_{a\ge1}a^{s_i+1}q^{-ah}\\
 &=O\bigl((n^{s_i}+1)q^{-h}\bigr)
 \longrightarrow0.
\end{align*}
Hence there is $n_{\mathrm{high}}$ such that, for every $n\ge n_{\mathrm{high}}$, the sum of the high-relative-degree probabilities over all positive-dimensional strata is less than $\delta/4$, uniformly in $P$.

Let
\[
 W_r=\coprod_{i=1}^N
 \coprod_{\substack{Q\in|Y_i|\\\deg Q\le r}}
 Q^{(2)}_{Y_i}.
\]
Each composite $Q^{(2)}_{Y_i}\to X_{\red}$ is an immersion with closed image, hence a closed immersion by
\cite[\href{https://stacks.math.columbia.edu/tag/01IQ}{Tag~01IQ}]{stacks-project}.
Thus $W_r$ is finite and closed in $X_{\red}$.
Let $n_{\mathrm{low}}$ be the integer supplied by \zcref{lem:semiample-restriction-closed} for $W_r$, and set
\[
 n_1=\max\{n_0,n_{\mathrm{high}},n_{\mathrm{low}}\}.
\]
Fix $n\ge n_1$.
Trivialize $\sheaf O_B(d)$ on each $P^{(2)}$ meeting $\rho(W_r)$, as in \zcref{lem:semiample-stable-image}.

For $P\in|B|$, put
\[
 W_{r,P}=\coprod_{\substack{i,Q\in|Y_i\cap\rho^{-1}(P)|\\\deg Q\le r}}Q^{(2)}_{Y_i}.
\]
The morphism $W_{r,P}\to B$ factors through $P^{(2)}$, since
$\rho^*(\mfrak_P^2)\subseteq\mfrak_Q^2$ on each component.
For $d\gg0$, restriction to $W_{r,P}$ therefore factors as
\[
 H^0(X_{\red},\sheaf O(nA+dE))\longrightarrow I_P(n)
 \longrightarrow H^0(W_{r,P},\sheaf O(nA)).
\]
The composite is surjective, so the second map is surjective.
Thus
\[
 I_P(n)\twoheadrightarrow
 \bigoplus_{\substack{i,Q\in|Y_i\cap\rho^{-1}(P)|\\\deg Q\le r}}
 H^0(Q^{(2)}_{Y_i},\sheaf O(nA)).
\]
The low-degree conditions are independent.
Their probability is
\[
 \prod_{i=1}^N
 \prod_{\substack{Q\in|Y_i\cap\rho^{-1}(P)|\\\deg Q\le r}}
 \bigl(1-q^{-(s_i+1)\deg Q}\bigr)
 \ge\prod_{i=1}^N\gamma_{i,P}\ge\delta.
\]  The union bound now gives
\[
 \beta_P(n)\ge\delta-\frac{\delta}{4}-\frac{\delta}{4}
 =\frac{\delta}{2}>0.
\]
The lower bound is independent of $P$ and holds for every $n\ge n_1$.
\end{proof}

\begin{theorem}
\label{thm:semiample-mult}
There is an integer $n_2$ such that, for every $n\ge n_2$, the density, as $d\to\infty$, of sections
\[
 s\in H^0(X,\sheaf O_X(nA+dE))
\]
whose zero scheme meets every $Y_i$ smoothly in codimension one is
\[
 \prod_{P\in|B|}\beta_P(n)>0.
\]
For every such section and every $\xi\in D_s$,
\begin{gather*}
 \ord_\xi(f_{s,\xi})=1,
 \qquad
 f_{s,\xi}\text{ is a nonzerodivisor on }\sheaf O_{X,\xi},
 \\
 \text{and}
 \quad
 e_\xi(D_s)=e_\xi(X).
\end{gather*}
\end{theorem}

\begin{proof}
Choose $n_2\ge n_1$ so that \zcref{lem:semiample-restriction-closed} applies to $W=X_{\red}$.
For $n\ge n_2$ and $d\gg0$, restriction to $X_{\red}$ is surjective, with fibers of equal cardinality.
The density is therefore \eqref{eq:semiample-product}, and it is positive by \zcref{lem:semiample-positive-factors}.
Apply \zcref{prop:closed-to-all-multiplicity} with $c=1$.
\end{proof}

\begin{corollary}
\label{cor:semiample-very-ample}
If $E$ is very ample, then, for $n\ge n_2$,
\[
 \prod_{P\in|B|}\beta_P(n)
 =\prod_{i=1}^N\zeta_{Y_i}(s_i+1)^{-1}.
\]
\end{corollary}

\begin{proof}
Since $E$ is very ample, $\pi$ is a closed immersion.
Thus $\rho:X_{\red}\to B$ is an isomorphism and
\[
 X_{P^{(2)}}\cong P^{(2)}
 \qquad
 \text{and}
 \qquad
 I_P(n)=H^0(P^{(2)},\sheaf O_{X_{\red}}(nA)).
\]
Let $P\in Y_i$.
The restriction
\[
 H^0(P^{(2)},\sheaf O(nA))
 \longrightarrow
 H^0(P^{(2)}_{Y_i},\sheaf O(nA))
\]
is induced by a quotient of local Artinian rings and is therefore surjective.
The target has cardinality $q^{(s_i+1)\deg P}$.
Its zero class is the only bad first jet.
Therefore
\[
 \beta_P(n)=1-q^{-(s_i+1)\deg P}.
\]
Multiplying over the closed points of the strata gives the formula; compare with \cite[Example~4.3]{erman-wood}.
\end{proof}

\appendix

\section{A finite construction of permissible stratifications}
\label{app:expanded-strat}

The proof of \zcref{thm:stratification} gives the following finite procedure.
Start with $W_0=X_{\red}$.
If $W_\nu=\varnothing$, stop.
Otherwise choose an irreducible component $Z_\nu$ of $W_\nu$, choose a nonempty open subset
\[
  U_\nu\subseteq\Reg(Z_\nu)
\]
which is disjoint from the other irreducible components of $W_\nu$ and such that $X$ is normally flat along $Z_\nu$ at every point of $U_\nu$, and put
\[
  W_{\nu+1}=(W_\nu\setminus U_\nu)_{\red}.
\]
The existence of $U_\nu$ is \zcref{thm:generic-nf} together with openness and nonemptiness of the regular locus (\cite[\href{https://stacks.math.columbia.edu/tag/07R2}{Tag~07R2}]{stacks-project}).
At each step, the underlying closed subset becomes strictly smaller.
Since $X$ is Noetherian, the procedure terminates.
The nonempty sets $U_\nu$ are the strata.

\section{Closed points and arbitrary points}
\label{app:closed-points}

Let $k$ be perfect, let $Y$ be smooth over $k$, let $\sheaf L_1,\ldots,\sheaf L_c$ be invertible on $Y$, and let $t_j\in H^0(Y,\sheaf L_j)$.
Put
\[
 V=V(t_1,\ldots,t_c).
\]
Along $V$, differentiation defines
\begin{equation}
\label{eq:first-jet-map}
 \delta_t:\bigoplus_{j=1}^c\sheaf L_j^{-1}|_V
 \longrightarrow\Omega^1_{Y/k}|_V.
\end{equation}
If $t_j=g_je_j$ for a local generator $e_j$ of $\sheaf L_j$, the $j$th summand sends $e_j^{-1}$ to $dg_j|_V$.

The following lemma shows that closed points detect transversality.

\begin{lemma}
\label{lem:rank-closed-points-all-points}
For $\xi\in V$, the rank of $\delta_t(\xi)$ equals the rank of the classes of $g_1,\ldots,g_c$ in
\[
 \mfrak_{Y,\xi}/\mfrak_{Y,\xi}^2.
\]
If this rank is $c$ at every closed point of $V$, then it is $c$ at every point of $V$.
\end{lemma}

\begin{proof}
If $e_j'=u_je_j$, then $g_j'=u_j^{-1}g_j$ and
\[
 dg_j'|_V=u_j^{-1}dg_j|_V;
\]
thus \eqref{eq:first-jet-map} is well defined.
For $\xi\in V$, the canonical map
\[
 \mfrak_{Y,\xi}/\mfrak_{Y,\xi}^2
 \longrightarrow
 \Omega^1_{Y/k}\otimes\kappa(\xi)
\]
is injective by
\cite[\href{https://stacks.math.columbia.edu/tag/00TU}{Tag~00TU}]{stacks-project}, because $\kappa(\xi)/k$ is separably generated.
It sends the class of $g_j$ to $dg_j(\xi)$, proving the first assertion.

In local frames, the rank $<c$ locus of \eqref{eq:first-jet-map} is cut out by the $c\times c$ minors of its matrix; it is therefore closed in $V$.
The scheme $V$ is Jacobson by
\cite[\href{https://stacks.math.columbia.edu/tag/02J6}{Tag~02J6}]{stacks-project}; a nonempty closed subset therefore contains a closed point.
The rank-defect locus is empty.
\end{proof}
Now we show that closed points suffice.

\begin{proposition}
\label{prop:closed-to-all-multiplicity}
Let $X$ be of finite type over a perfect field $k$, let
\[
 X_{\red}=Y_1\sqcup\cdots\sqcup Y_N
\]
be a smooth permissible stratification, and put $s_i=\dim Y_i$.
Let $\sheaf L_1,\ldots,\sheaf L_c$ be invertible on $X$, with sections $t_j$, and set
\[
 V_i=Y_i\cap V(t_1,\ldots,t_c).
\]
Assume, for every $i$, that
\begin{enumerate}[label=\textup{(\alph*)}]
\item if $c\le s_i$, then $V_i$ is smooth of dimension $s_i-c$ at every closed point;
\item if $c>s_i$, then $V_i$ has no closed points.
\end{enumerate}
Then $V_i$ is smooth of dimension $s_i-c$ when $c\le s_i$, and $V_i=\varnothing$ when $c>s_i$.
If $\xi\in V(t_1,\ldots,t_c)$ and $f_j\in\sheaf O_{X,\xi}$ are local equations of the $t_j$, then $f_1,\ldots,f_c$ is $\sheaf O_{X,\xi}$-regular and
\[
 e\bigl(\sheaf O_{X,\xi}/(f_1,\ldots,f_c)\bigr)
 =e(\sheaf O_{X,\xi}).
\]
For $c=1$, one also has $f_1\in\mfrak_{X,\xi}\setminus\mfrak_{X,\xi}^2$.
\end{proposition}

\begin{proof}
Suppose $c\le s_i$.
At every closed point of $V_i$, \zcref{lem:smooth-first-jets} gives rank $c$ for \eqref{eq:first-jet-map}.
By \zcref{lem:rank-closed-points-all-points}, its rank is $c$ on $V_i$.
Hence \zcref{lem:smooth-first-jets} shows that $V_i$ is regular of pure dimension $s_i-c$.
It is smooth over $k$ by
\cite[\href{https://stacks.math.columbia.edu/tag/0B8X}{Tag~0B8X}]{stacks-project}.

If $c>s_i$, then $V_i$ has no closed point.
Since it is of finite type over $k$, it is empty by
\cite[\href{https://stacks.math.columbia.edu/tag/02J6}{Tag~02J6}]{stacks-project}.

Let $\xi$ lie in the common zero scheme and choose $i$ with $\xi\in Y_i$.
The preceding paragraph gives $c\le s_i$.
The classes of $f_1,\ldots,f_c$ in
\[
 \mfrak_{Y_i,\xi}/\mfrak_{Y_i,\xi}^2
\]
are linearly independent.
The local prime defining $(\overline{Y_i})_{\red}$ is permissible, so \zcref{thm:bridge-ci} applies.
\end{proof}

\section{The infinite-field theorem}
\label{sec:infinite-field}

We use the results of \zcref{sec:bridge} and the axiomatic approach of \cite{spreafico} to prove the analogous Bertini theorem for Hilbert--Samuel multiplicity over any \emph{infinite} field $k$.
The result for $k$ algebraically closed can be found in \cite{tommaso} and \cite{complex-links}, though we remove the requirement of those results that the scheme is either reduced or equidimensional.

\begin{theorem}
\label{thm:infinite-field}
Let $k$ be an infinite field, let $r\ge1$, let $X$ be of finite type over $k$, and let
\[
 X_{\red}=Y_1\sqcup\cdots\sqcup Y_N
\]
be a smooth permissible stratification.
Let $\phi:X\to\bP^r_k$ be a morphism.
Assume that, for every $i$ and every $\xi\in Y_i$, the extension
\[
 \kappa(\xi)/\kappa(\phi(\xi))
\]
is separably generated.
Then there is a nonempty open subset
\[
 U\subseteq(\bP^r_k)^\vee
\]
such that, for every $H\in U(k)$ and every $i$, the scheme $Y_i\cap\phi^{-1}(H)$ is either empty or smooth of pure codimension one in $Y_i$.
If $\xi\in\phi^{-1}(H)$ and $f_{H,\xi}$ is a local equation of the pullback hyperplane, then
\begin{gather*}
 \ord_\xi(f_{H,\xi})=1,
 \qquad
 f_{H,\xi}\text{ is a nonzerodivisor on }\sheaf O_{X,\xi},\\
 \text{and}
 \quad
 e_\xi(\phi^{-1}(H))=e_\xi(X).
\end{gather*}
\end{theorem}

\begin{proof}
Fix $i$.
Since $Y_i$ is smooth over $k$, it is geometrically regular over $k$ by
\cite[\href{https://stacks.math.columbia.edu/tag/038X}{Tag~038X}]{stacks-project}.
The morphism $\phi_i=\phi|_{Y_i}$ is residually separable in the sense of
\cite[\S\,1.1]{spreafico}: indeed, every residue-field extension of $\phi_i$ is separably generated by hypothesis, hence separable by
\cite[\href{https://stacks.math.columbia.edu/tag/030X}{Tag~030X}]{stacks-project}.
Regularity satisfies Spreafico's axioms by \cite[Proposition~2.4]{spreafico}.
Thus
\cite[Corollary~4.3]{spreafico} gives a dense open subset
\[
 U_i'\subseteq(\bP^r_k)^\vee
\]
such that $Y_i\cap\phi^{-1}(H)$ is geometrically regular for every $H\in U_i'(k)$.

Let $W_i=\overline{\phi(Y_i)}_{\red}$.
Hyperplanes containing $W_i$ form a proper closed subset $T_i\subseteq(\bP^r_k)^\vee$: if $\dim W_i=0$, they are the hyperplanes through its unique closed point, and if $\dim W_i>0$, the same assertion follows because $W_i$ is not contained in every hyperplane.
For $H\notin T_i$, the pullback of an equation of $H$ is nonzero at the generic point of the integral scheme $Y_i$.
Hence its zero scheme on $Y_i$ is either empty or an effective Cartier divisor, and therefore is pure of codimension one.
For $H\in U_i'$, it is geometrically regular, hence smooth over $k$ by
\cite[\href{https://stacks.math.columbia.edu/tag/038X}{Tag~038X}]{stacks-project}.

Set
\[
 U=\bigcap_{i=1}^N(U_i'\setminus T_i).
\]
This is a nonempty open subset of $(\bP^r_k)^\vee$.
Since $k$ is infinite, $U(k)\ne\varnothing$; equivalently, $k$-rational points are Zariski dense in projective space.

Let $H\in U(k)$ and $\xi\in Y_i\cap\phi^{-1}(H)$.
The Cartier divisor $Y_i\cap\phi^{-1}(H)$ is smooth of codimension one at $\xi$.
Hence its local equation has nonzero image in
\[
 \mfrak_{Y_i,\xi}/\mfrak_{Y_i,\xi}^2.
\]
By \zcref{thm:bridge}, its image in $\gr_{\mfrak_\xi}(\sheaf O_{X,\xi})$ is regular of degree one.
Therefore
\begin{gather*}
 \ord_\xi(f_{H,\xi})=1,
 \qquad
 f_{H,\xi}\text{ is a nonzerodivisor},\\
 \text{and}
 \quad
 e_\xi(\sheaf O_{X,\xi}/(f_{H,\xi}))=e_\xi(X).\qedhere
\end{gather*}
\end{proof}

\begin{remark}
\label{rem:infinite-imperfect}
The smoothness hypothesis on the strata cannot in general be replaced by regularity when $k$ is imperfect.
\cite[Corollary~4.3]{spreafico} applied to regularity assumes geometric regularity of the source.
For a finite-type $k$-scheme, this is equivalent to smoothness over $k$ by
\cite[\href{https://stacks.math.columbia.edu/tag/038X}{Tag~038X}]{stacks-project}.
If $k$ is perfect, every regular finite-type $k$-scheme is smooth, so \zcref{thm:stratification} supplies the required stratification.
\end{remark}

\bibliographystyle{amsalpha}
\bibliography{ref}

\end{document}